\documentclass[12pt]{amsart}
\usepackage{geometry}               
\usepackage{amssymb}
\usepackage{mathrsfs}
\usepackage{cite}
\usepackage{amsmath}
\usepackage{graphicx} 
\usepackage{calc}
\usepackage{hyperref}
\usepackage{amsthm}
\usepackage{amssymb}
\usepackage[english]{babel}
\usepackage{indentfirst}

\theoremstyle{definition}
\newtheorem{theorem}{Theorem}[section]
\newtheorem{definition}[theorem]{Definition}
\newtheorem{corollary}[theorem]{Corollary}
\newtheorem{lemma}[theorem]{Lemma}
\newtheorem{proposition}[theorem]{Proposition}
\newtheorem{remark}[theorem]{Remark}

\title[$C^{\infty}$ closing for Hamiltonian flows on symplectic $4$-manifolds]{On the $C^{\infty}$ closing lemma for Hamiltonian flows on symplectic $4$-manifolds}
\author{Dong Chen}

\address{Dong Chen: Ohio State University, Department of Mathematics, Columbus, OH 43210, USA}
\email{chen.8022@osu.edu}
\date{}

\begin{document}

\maketitle
\begin{abstract}
The main result in this paper is the $C^{\infty}$ closing lemma for a large family of Hamiltonian flows on $4$-dimensional symplectic manifolds, which includes classical Hamiltonian systems. First we prove the $C^{\infty}$ closing lemma and the $C^r$ general density theorem for geodesic flows on closed Finsler surfaces by combining a result of Asaoka-Irie with the dual lens map technique. Then we extend our results to Hamiltonian flows with certain restriction. We also list some applications of our results in differential geometry and contact topology.
\end{abstract}

\section{Introduction}
\footnote{2010 \emph{Mathematics Subject Classification.} 58B20, 37J55, 53D35.

\emph{Key words and phrases.} $C^{\infty}$ closing lemma, Finsler metric, duel lens map, Reeb flow, Hamiltonian flow, perturbation.
}

The $C^r (r\geq 2)$ closing lemma was put forward by Smale (\cite{S98}, Problem 10) as one of the most important open problems for this century. Given any non-wandering point $v$ of a dynamical system, either a diffeomorphism or a flow, the goal of the $C^r$ closing lemma is to make a $C^r$ pertubation of the dynamic to close the orbit which starts at $v$.

The history of the closing problem dates back to Poincar\'{e} \cite{P92}. It has been actively studied since the pioneering work of Pugh \cite{P67b}\cite{P67a}, who established the $C^1$ closing lemmas for flows and diffeomorphisms. The $C^1$ closing lemmas for symplectic diffeomorphisms and Hamiltonian flows were proved later by Pugh-Robinson \cite{PR83}. In higher smoothness $r\geq 2$, despite of the strong interest in the closing problem, the $C^r$ closing lemma was only known in a few special cases \cite{G00}\cite{AZ12}. A celebrated counterexample constructed by Herman\cite{H91} disproves the $C^{\infty}$ closing lemma for Hamiltonian flows.  Recently Asaoka-Irie\cite{AI16} proved the $C^{\infty}$ closing lemma for Hamltonian diffeomorphisms of closed surfaces by applying the ECH (embedded contact homology) techniques developed by Cristofaro-Gardiner, Hutchings and Ramos \cite{CHR15}.

As we turn our attention to Riemannian geodesic flows, only the $C^0$ closing lemma \cite{R12} was confirmed. The major difficulty in the setting of Riemannian geodesic flows is one cannot make a local perturbation without breaking the Riemannian structure. However, the Finsler counterparts are more flexible \cite{C17b}, allowing us to make perturbations in an arbitrarily small neighborhood.

For any $2\leq r\leq \infty$, we use $\mathcal{F}^{r}(\Sigma)$ to denote the set of all $C^r$ smooth Finsler metrics on $\Sigma$, equipped with $C^r$ topology. Denote by $\mathcal{F}_R^{r}(\Sigma)$ the subspace of reversible $C^r$ smooth Finsler metrics. A subset in $\mathcal{F}^{r}(\Sigma)$ (resp. $\mathcal{F}_R^{r}(\Sigma)$) is called \textit{residual} if it contains a countable intersection of open dense sets. Denote by $U_{\varphi}M$ the unit tangent bundle of $(M,\varphi)$. 

\begin{theorem}\label{clslem}
Let $\Sigma$ be a closed surface. For any $\varphi\in\mathcal{F}^{\infty}(\Sigma)$(resp. $\varphi\in\mathcal{F}_R^{\infty}(\Sigma)$) and any $v\in U_\varphi \Sigma$, there exists a $C^{\infty}$-small perturbation $\tilde{\varphi}\in\mathcal{F}^{\infty}(\Sigma)$(resp. $\tilde{\varphi}\in\mathcal{F}_R^{\infty}(\Sigma)$) of $\varphi$, such that the geodesic in $(\Sigma,\tilde{\varphi})$ which starts at $v$ is closed.
\end{theorem}
Note that the phase space of a geodesic flow is a symplectic manifold with exact symplectic form, which is not the case for Hermann's counterexample \cite{H91}.
Thus our results  provides evidence that the $C^{\infty}$ closing lemma should hold for Hamiltonian flows on exact symplectic manifolds.

To prove Theorem \ref{clslem} we first prove the following localized closing lemma:
\begin{theorem}\label{locclslem}
Let $\Sigma$ be a closed surface. For any $\varphi\in\mathcal{F}^{\infty}(\Sigma)$ (resp. $\varphi\in\mathcal{F}_R^{\infty}(\Sigma)$) and any non-empty open set $U\subseteq U_\varphi \Sigma$, $\varphi$ admits a $C^{\infty}$ small perturbation in $\mathcal{F}^{\infty}(\Sigma)$ (resp. $\varphi\in\mathcal{F}_R^{\infty}(\Sigma)$) such that the resulting geodesic flow has a periodic orbit passing through $U$.
\end{theorem}

Theorem \ref{locclslem} implies the $C^r(r\geq 2)$ general density theorem:
\begin{theorem}\label{dencor}
For any closed surface $\Sigma$ there exists a residual set $\mathcal{P}\subset \mathcal{F}^{r}(\Sigma)$ (or $\mathcal{F}_R^{r}(\Sigma)$) such that the periodic geodesics of any $\varphi\in\mathcal{P}$ form a dense set in $U_{\varphi}\Sigma$.
\end{theorem}
\begin{proof}
Using the standard Baire argument together with Theorem \ref{locclslem} we can prove the case $r=\infty$. For $r< \infty$ notice that $\mathcal{F}^{\infty}(\Sigma)$ is a dense subset in $\mathcal{F}^{r}(\Sigma)$.
\end{proof}

Theorem \ref{clslem} also holds for a more general family of Hamitonian flows. Let $H: \Omega\to\mathbb{R}$ be a smooth Hamiltonian on a symplectic manifold $(\Omega, \omega)$ and let $\Phi^t_H$ be the corresponding Hamitonian flow. We say that the level set $H^{-1}(h)$ is {\it regular} if $h$ is a regular value of $H$. A regular level set $H^{-1}(h)$ is {\it of contact type} (see \cite{W79}) if there exists a $1$-form $\lambda$ on $H^{-1}(h)$ such that $d\lambda=\omega$ and $\lambda(\dot{\Phi}_H^t(x))\neq 0$ for any $x\in H^{-1}(h)$. 
 
\begin{theorem}\label{hamclosing}
Let $(\Omega, \omega)$ be a symplectic $4$-manifold and $H: \Omega\to\mathbb{R}$ be a smooth Hamiltonian. Assume that the regular level set $H^{-1}(h)$ is compact and of contact type. Then for any $v\in H^{-1}(h)$, one can make a $C^{\infty}$-small perturbation $\tilde{H}$ of $H$ such that the $\Phi^t_{\tilde{H}}$-orbit through $v$ is closed.
\end{theorem}

It is worth noting that not all hypersurfaces in $(\Omega,\omega)$ are of contact type \cite{W79} \cite{HZ11}. Nevertheless the assumption is flexible enough to contain a large family of Hamiltonians. For instance, if a Hamiltonian system is based on classical mechanics or differential geometry, then all of its regular compact level sets are of contact type. See Section 5 for a more detailed discussion.

Being of contact type is independent of the choice of $H$. In fact, if $H_1$ and $H_2$ have a common level set $Y$, then the orbits of their associated Hamiltonian flows on $Y$ coincide. Thus we can define periodic orbits on $Y$ without picking a particular Hamiltonian. Weinstein conjectured in \cite{W79} that any closed hypersurface $Y$ of contact type has at least one periodic orbit. This conjecture was proved by Viterbo \cite{V87} for hypersurfaces in $\mathbb{R}^{2n}$, by Hofer\cite{H93} for $Y=S^3$ and by Taubes \cite{T07} for $\dim Y=3$. Since every hypersurface $Y\subseteq \Omega$ of contact type can be represented as a regular level set of some Hamiltonian $H$, by applying the Baire argument again we get the following corollary of Theorem \ref{hamclosing}:

\begin{corollary}
Let $(\Omega, \omega)$ be a symplectic $4$-manifold and let $Y$ be a hypersurface of contact type. A $C^{\infty}$ generic closed hypersurface near $Y$ has a dense subset consisting of periodic orbits. 
\end{corollary}

This paper is organized as follows: in Section 2 we introduce background and known results regarding Finsler manifolds and Reeb flows. The dual lens maps technique is presented in Section 3 with applications in perturbing Finsler metrics and Hamiltonian flows. Section 4 is devoted to prove the main theorems via tools from Section 2 and 3. One can find interesting applications of our results in both differential geometry and Hamiltonian systems in Section 5. 

\subsection*{Acknowledgement}
The author would like to thank Federico Rodriguez Hertz for bringing up this question. The author is also grateful to Dmitri Burago, Leonid Polterovich, Daniel Thompson, Andrey Gogolev, Amie Wilkinson, Aaron Brown, Andre Neves, Jeff Xia and Keith Burns for their useful comments and remarks.

\section{Preliminaries}
\subsection{Hamiltonian flows on symplectic manifolds} Let $(\Omega,\omega)$ be a symplectic manifold and let $H$ be a smooth function on $\Omega$. The \textit{Hamiltonian vector field $X_H$} on $\Omega$ is defined to be the unique solution to the equation $\omega(X_H, V)=dH(V)$ for any smooth vector field $V$ on $\Omega$. The vector field $X_H$ is well-defined due to non-degeneracy of $\omega$. The flow $\Phi_H^t$ on $\Omega$ defined by $\dot{\Phi}_H^t=X_H$ is called \textit{the Hamiltonian flow on $\Omega$ with Hamiltonian $H$}. One can easily verify that $\Phi_H^t$ preserves $\omega$ and hence the volume form $\omega^n$.

\subsection{Finsler manifolds and geodesics.} A \textit{Finsler metric} $\varphi$ on $M$ is a smooth family of quadratically convex norms on each tangent space. It is \textit{reversible} if $\varphi(v)=\varphi(-v)$ for all $v\in TM$.  If $\gamma:[a,b]\rightarrow M$ is a smooth curve on a Finsler manifold $(M, \varphi)$, the length of $\gamma$ can be defined by
$$L(\gamma):= \int^b_a \varphi(\gamma'(t))dt.$$
We define a quasimetric $d: M\times M\to \mathbb{R}$ by setting $d(x,y)$ to be the infimum of the lengths of all piecewise smooth curves starting from $x$ and ending at $y$. It is clear that $d$ satisfies the triangle inequality
$$d(x,z)\leq d(x,y)+d(y,z), \forall x,y,z\in M,$$ 
but $d$ may be non-symmetric since $d(x,y)$ may not be equal to $d(y,x)$ if $\varphi$ is not reversible. For any $x\in M, r>0$, we define the \textit{geodesic balls  $B^{\pm}_{\varphi}(x,r)$} to be 
$$B^+_{\varphi}(x,r):=\{y\in M: d(x,y)<r\}, B^-_{\varphi}(x,r):=\{y\in M: d(y,x)<r\}.$$
If $\varphi$ is reversible then $B^+_{\varphi}(x,r)=B^-_{\varphi}(x,r)$ for all $x$ and $r$. A unit-speed curve $\gamma: [a, b]\rightarrow M$ is said to be a \textit{geodesic} of $(M, \varphi)$ if for every sufficiently small interval $[c, d]\subseteq [a, b]$,  $\gamma|_{[c,d]}$ realizes the distance from $\gamma(c)$ to $\gamma(d)$. For any $v\in U_{\varphi}M$, there exists a unique geodesic $\gamma_v^\varphi: (a,b)\to M$ such that $\dot{\gamma}_v^\varphi(0)=v.$ We define the \textit{geodesic flow} $g_t^{\varphi}: U_{\varphi}M\to U_{\varphi}M$ to be $g_t^{\varphi}(v):=\dot{\gamma}_v^\varphi(t)$. Given $v\in U_{\varphi}M$ and $r>0$, we define \textit{forward geodesic disc $D^+_{\varphi}(v,r)$} and \textit{backward geodesic disc $D^-_{\varphi}(v,r)$} to be 
$$D^{\pm}_{\varphi}(v,r):=\overline{B^{\mp}_{\varphi}(\pi(g_{\pm r}^{\varphi}(v)), r)}.$$ 

\begin{figure}[htb] 
\centering\includegraphics[width=3.5in]{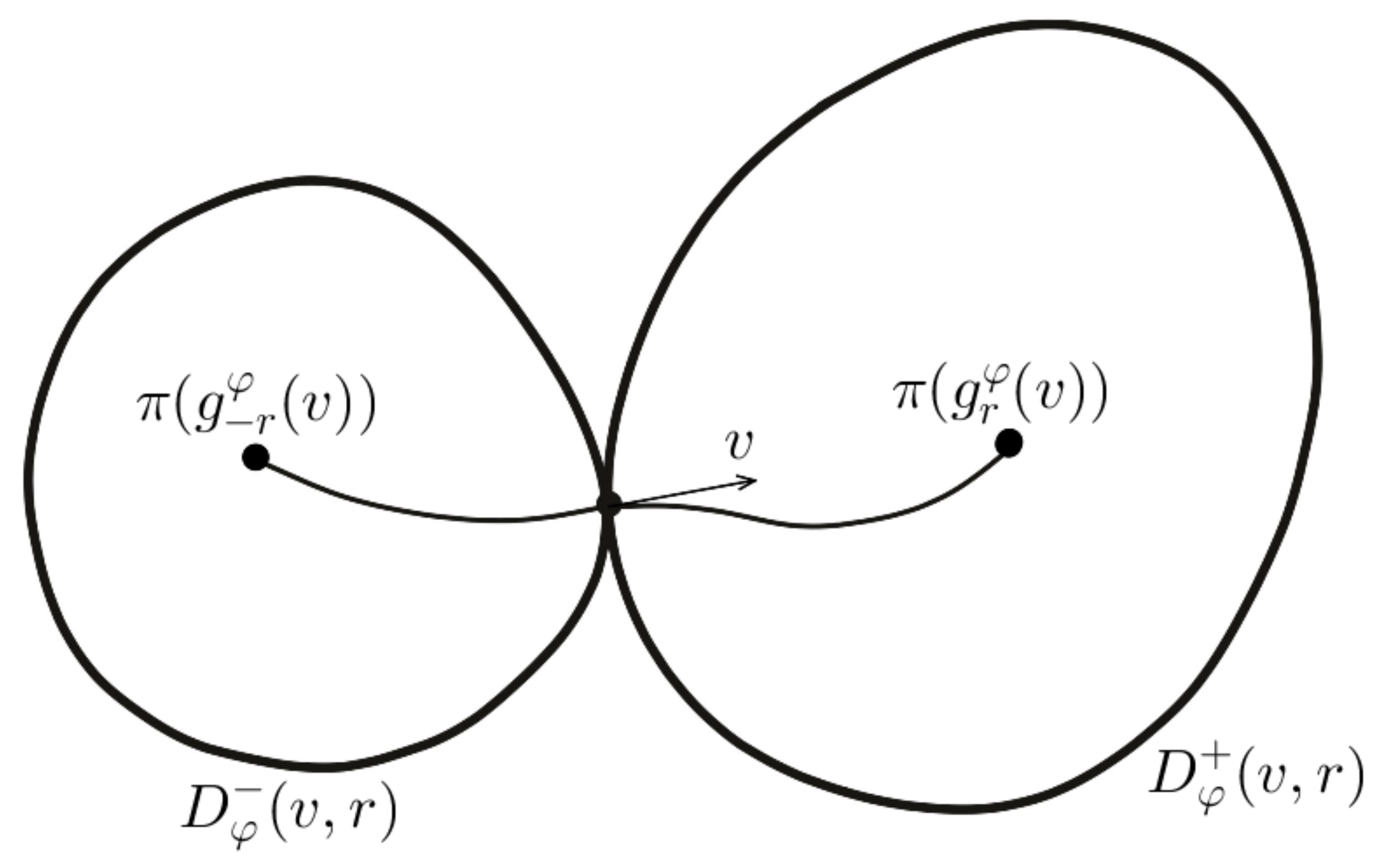} 
\caption{}\label{fig:1} 
\end{figure}

The following lemma is a simple corollary of the triangle inequality:
\begin{lemma}\label{distlem}
For all $v\in U_{\varphi}M $ and sufficiently small $r>0$, we have $D^{-}_{\varphi}(v,r)\cap D^{+}_{\varphi}(v,r)=\{\pi(v)\}$.
\end{lemma} 
\subsection{Geodesic flow on cotangent bundles} Another version of geodesic flow is defined on the cotangent bundle $T^*M$. Notice that $T^*M$ is symplectic with a natural symplectic form $\omega=\sum dq_i\wedge dp_i$. For any Finsler manifold $(M,\varphi)$, we define the \textit{dual norm} on cotangent bundle $T^*M$ by
$$\varphi^*(\alpha):=\sup_{v\in U_{\varphi}M}\{\alpha(v)\}, \text{ for }\alpha\in T^* M.$$
\textit{The geodesic flow $g_t^{\varphi}$ on $T^*M$} is defined to be the Hamiltonian flow on the cotangent bundle $T^*M$ with Hamiltonian $(\varphi^*)^2/2$.  We denote by $U^*_{\varphi}M$ the unit cotangent bundle of $(M,\varphi)$. Since the Hamiltonian $H$ is invariant under Hamiltonian flow $\Phi_H^t$, the geodesic flow $g_t^{\varphi}$ leaves $U^*_{\varphi}M$ invariant. From now on we will only consider the geodesic flow on $U^*_{\varphi}M$ instead of $T^*M$.

The \textit{Legendre transform $\mathscr{L}$ on $U_{\varphi}M$} is defined in the following way: for any tangent vector $v\in U_{\varphi}M$, its Legendre transform $\mathscr{L}(v)$ is the unique covector $\alpha\in U_{\varphi}^*M$ such that $\alpha(v)=1$. The Legendre transform $\mathscr{L}$  conjugates the  geodesic flow and its cotangent version. In this paper we abuse the notation and use $g_t^{\varphi}$ to denote both geodesic flows.

\subsection{Reeb flows on contact manifolds}
Let $\lambda$ be a 1-form on a $(2n-1)$-dimensional orientable manifold $Y$. $\lambda$ is called a \textit{contact form} if $\lambda\wedge(d\lambda)^{n-1}$ does not vanish. A smooth manifold $Y$ with a contact form $\lambda$ is called a \textit{contact manifold}. The \textit{Reeb vector field} $R_{\lambda}$ is defined as the unique vector field satisfying:
$$\lambda(R_{\lambda})=1, d\lambda(R_{\lambda}, \cdot)=0.$$
The flow on $(Y,\lambda)$ generated by $R_{\lambda}$ is called the \textit{Reeb flow on $(Y,\lambda)$}.

Let $\Phi^t_H$ be a Hamitonian flow on $(\Omega, \omega)$ and let $H^{-1}(h)$ be a regular level set of contact type with $1$-form $\lambda$. By definition $(H^{-1}(h), \lambda)$ is a contact manifold and the correponding Reeb flow is a smooth time-change of the restriction of $\Phi^t_H$ to $H^{-1}(h)$. In particular, for any Finsler manifold $(M,\varphi)$, $U^*_{\varphi}M$ is a contact manifold with  $\lambda=\sum p_idq_i$. The Reeb flow on $(U^*_{\varphi}M,\lambda)$ coincides with the geodesic flow $g_t^{\varphi}$. 

When dim $Y=3$, a remarkable result from Cristofaro-Gardiner, Hutchings and Ramos \cite{CHR15} shows that the volume of $Y$ is determined by its embedded contact homology (ECH). By applying this result to perturbation of Reeb flows, Irie \cite{I15}(see also \cite{AI16}) proved the following lemma:
\begin{lemma}(\cite{AI16} Lemma 2.1)\label{echlem}
Let $(Y,\lambda)$ be a closed contact three-manifold. For any $h\in C^{\infty}(Y,\mathbb{R}_{\geq 0})\backslash\{0\}$, there exist $t\in [0,1]$ and $\gamma\in\mathcal{P}(Y,(1+th)\lambda)$ which intersects supp $ h$.
\end{lemma}
\begin{remark}
In \cite{I15} Lemma 4.1, Irie proved a weaker version of local closing lemma: given any open set $U$ in a Riemannian surface $(M,g)$, one can make a conformal perturbation of $g$ so that there exists a closed geodesic in the perturbed metric passing through $U$. The argument is a direct application of the above lemma by taking $h\in C^{\infty}(M,\mathbb{R}_{\geq 0})\backslash\{0\}$ supported on $U$.

However, unlike Theorem \ref{locclslem}, this result does not have any control in the direction of the closed geodesic. It is tempting to try and prove Theorem \ref{locclslem} by taking a $C^{\infty}$-small function $h$ supported on an open set in $U_{\varphi} M$ and carrying Irie's argument. However in this case the function $(1+h)\varphi$ is no longer a norm on $TM$ hence does not generate a Reeb flow. Therefore the proof of Theorem \ref{locclslem} is not a direct generalization of Irie's argument. In order to construct the perturbed Finsler metric we need the tools in Section 3.
\end{remark}

\section{$C^{\infty}$ Perturbation of simple Finsler discs and Hamiltonian flows}
\subsection{Simple Finsler discs and dual lens maps} 
A Finsler metric $\varphi$ on an $n$-dimensional disc $D$ is called \textit{simple} if it satisfies the following three conditions:

(1) Every pair of points in $D$ is connected by a unique geodesic.

(2) Geodesics depend smoothly on their endpoints.

(3) The boundary is strictly convex, that is, geodesics never touch it at their interior points.

It was proven by Whitehead \cite{W32} (see also Bao-Chern-Shen \cite{BCS00}) that any Finsler manifold is locally simple. In particular, when $M$ is compact, we have

\begin{lemma}\label{radlem}
For any point $x$ in a compact Finsler manifold $(M,\varphi)$, there exists $\rho>0$ such that for any $x\in M, r\leq \rho$, $(\overline{B^{\pm}_{\varphi}(x,r)}, \varphi)$ are simple. 
\end{lemma}

Once $(D,\varphi)$ is simple, denote by $U_{in}, U_{out}$ the set of inward, outward pointing unit tangent vectors with base points in $\partial D$ respectively. We consider subsets $U^*_{in}=\mathscr{L}(U_{in})$ and $U^*_{out}=\mathscr{L}(U_{out})$ of $U^*_{\varphi}D$, where $\mathscr{L}$ is the Legendre transform defined in Section 2.3.

For any $\alpha\in U^*_{in}$, denote by $\sigma(\alpha)$ the first intersection of $U^*_{out}$ with the forward $g_t^{\varphi}$ orbit of $\alpha$. This defines a map $\sigma: U^*_{in}\rightarrow U^*_{out}$, called \textit{the dual lens map of $\varphi$}. If $\varphi$ is reversible, then the dual lens map is reversible, namely, $-\sigma(-\sigma(\alpha))=\alpha$ for every $\alpha\in U^*_{in}$. Note that $U^*_{in}$ and $U^*_{out}$ are $(2n-2)$-dimensional submanifolds of $T^*D$. The restriction of the canonical symplectic 2-form of $T^*D$ to $U^*_{in}$ and $U^*_{out}$ determines the symplectic structure. Moreover, the dual lens map $\sigma$ is symplectic. 
\subsection{Perturbation of simple metrics on $D^2$}

\begin{definition}
For any symplectic map $\sigma: U^*_{in}\to U^*_{out}$ we  define two maps $P_{\sigma}: U^*_{in}\rightarrow S\times S$ and $Q_{\sigma}: U^*_{out}\rightarrow S\times S$ by 
$$P_{\sigma}(\alpha)=(\pi(\alpha), \pi(\sigma(\alpha)))$$
and
$$Q_{\sigma}(\beta)=(\pi(\sigma^{-1}(\beta)), \pi(\beta)),$$
here $\pi: T^*D\rightarrow D$ is the bundle projection and $S:=\partial D$. 
\end{definition}
\begin{remark}
Note that $P_{\sigma}=Q_{\sigma}\circ\sigma$ and both maps are bijection. The map $P^{-1}_{\sigma}$ ($Q^{-1}_{\sigma}$ resp.) takes two different points on the boundary $S$ and reports the inwards (outwards resp.) covector of the geodesic connecting these two points.
\end{remark}
\begin{definition}
Let $\Delta:=\{(x,x)\in S\times S: x\in S\}$. We define a 1-form $\lambda_{\sigma}$ on $S\times S\backslash\Delta$ as follows. For $p,q\in S, p\neq q, \xi\in T_p S, \eta\in T_q S$, define
$$\lambda_{\sigma}(\xi,\eta):=-P^{-1}_{\sigma}(p,q)(\xi)+Q^{-1}_{\sigma}(p,q)(\eta).$$
\end{definition}

In \cite{BI16}, the following theorem is proven, saying that under certain natural restrictions, a symplectic perturbation of $\sigma$ is the dual lens map of some Finsler metric closed to $\varphi$:

\begin{theorem}[\cite{BI16}, Remark 4.4]\label{prop1}
Let $\sigma$ be the dual lens map of a simple $C^{\infty}$ Finsler metric $\varphi$ on $D^2$. Let $W$ be the complement of a compact set in $U^*_{in}$ and  $\tilde{\sigma}$ is a symplecitc perturbation of $\sigma$ with $\tilde{\sigma}|_W=\sigma|_W$. Then $\tilde{\sigma}$ is the dual lens map of a simple $C^{\infty}$ Finsler metric $\tilde{\varphi}$ which coincides with $\varphi$ in some neighborhood of $\partial D$ if and only if
$$\int_{\{p\}\times\{S\backslash p\}}\lambda_{\tilde{\sigma}}=0 $$
for some (and then all) $p\in S$. The choice of $\tilde{\varphi}$ can be made in such a way that $\tilde{\varphi}$ converges to $\varphi$ in $C^{\infty}$ whenever $\tilde{\sigma}$ converges to $\sigma$ in $C^{\infty}$. In addition, if $\varphi$ is a reversible Finsler metric and $\tilde{\sigma}$ is reversible then $\tilde{\varphi}$ can be chosen reversible as well.
\end{theorem}

\begin{proposition}\label{maincor}
Let $\sigma$ be the dual lens map of a simple $C^{\infty}$ Finsler metric on $D^2$. Let  $U\subseteq U^*_{in}$ be a sufficiently small open neighborhood of $\alpha\in U^*_{in}$. Given any symplectic perturbation $\tilde{\sigma}$ of $\sigma$ with $\tilde{\sigma}=\sigma$ on $U^*_{in}\backslash U$, there exists a simple $C^{\infty}$ Finsler metric $\tilde{\varphi}$ whose  dual lens map is $\tilde{\sigma}$. Convergence and reversible cases are the same as those in Theorem \ref{prop1}.
\end{proposition}
\begin{proof}
Take $p\in \partial D^2\backslash\pi(U)$ and denote by $(U^*_{in})_p:=\pi^{-1}(p)\cap U^*_{in}$. 
Notice that on $(U^*_{in})_p$ we have $\sigma=\tilde{\sigma}$ hence $P_{\sigma}=P_{\tilde{\sigma}}$. Thus $P^{-1}_{\sigma}(p,\cdot)=P^{-1}_{\tilde{\sigma}}(p,\cdot)$. Similarly we have $Q^{-1}_{\sigma}(p,\cdot)=Q^{-1}_{\tilde{\sigma}}(p,\cdot)$, therefore $\lambda_{\sigma}(\xi, \cdot)=\lambda_{\tilde{\sigma}}(\xi, \cdot)$ for all $\xi\in (U^*_{in})_p$. By Theorem \ref{prop1},
$$\int_{\{p\}\times\{S\backslash p\}}\lambda_{\tilde{\sigma}}=\int_{\{p\}\times\{S\backslash p\}}\lambda_{{\sigma}}=0.$$
By applying Theorem \ref{prop1} again we finish the proof.
\end{proof}
\begin{remark}
Proposition \ref{maincor} remains valid for higher dimensions. It is an easy corrollary of Theorem 2 in \cite{BI16}.
\end{remark}
\begin{remark}\label{localpert}
The perturbation in Proposition \ref{maincor} is a local perturbation. Namely, the perturbation of $\tilde{\varphi}-\varphi$ can be confined to an arbitrarily small open neighborhood  of $\alpha$, if the support of $\tilde{\sigma}-\sigma$ is sufficiently small.
\end{remark}

\subsection{Local perturbation of Hamiltonian flows}
In order to make local perturbation of $\Phi_H^t$ we need the following classical result (see \cite{AM87} Theorem 5.2.19):
\begin{theorem}[Hamiltonian flow box theorem]\label{flowbox}
Let $(\Omega, \omega)$ be a symplectic manifold of dimension $\dim \Omega=2n$. For any point $x\in\Omega$ and any Hamiltonian function $H$ on $\Omega$ with $X_H(x)\neq 0$, one can find an open neighborhood $U$ of $x$ and coordinates $(q_1,  ..., q_n, p_1,..., p_n)$ in $U$ such that 
$$\omega=\sum_{i=1}^n dq_{i}\wedge dp_i, H|_U=p_n.$$
\end{theorem}

Namely, under this coordinates the flow in $U$ is given by 
$$\dot{p}_1=\cdots =\dot{p}_n= \dot{q}_1=\cdots =\dot{q}_{n-1}=0, \dot{q}_{n}=1.$$

It is clear that the condition $X_H(x)\neq 0$ can be fulfilled once we assume $h=H(x)$ is a regular value. Now we fix a regular value $h$. For any $x\in H^{-1}(h)$ and $T>0$, let $\Sigma_0, \Sigma_T\subseteq H^{-1}(h)$ be two transversal sections containing $x$ and $\Phi^T_H(x)$ respectively. We denote by $\mathscr{P}_T: \Sigma_0\to \Sigma_T$  the Poincar\'{e} map between these sections with respect to $\Phi^t_H$. Similar to the dual lens map, the Poincar\'{e} map $\mathscr{P}_T$ is symplectic with respect to the restriction of $\omega$ on $\Sigma_0$ and $\Sigma_T$. The following proposition shows that any local perturbation of $\mathscr{P}_T$ near $x$ can be realized as the Poincar\'{e} map of a perturbed Hamiltonian flow:


\begin{proposition}\label{pertham}
Let $V$ be a sufficiently small open neighborhood of $x$ in $\Sigma_0$ and $\tilde{\mathscr{P}}_T: \Sigma_0\to \Sigma_T$ be a symplectic perturbation of $\mathscr{P}_T$ on $\overline{V}$. There exists perturbation $\tilde{H}$ of $H$ such that $\tilde{\mathscr{P}}_T$ is the Poincar\'{e} map with respect to $\Phi^t_{\tilde{H}}$. The choice of $\tilde{H}$ can be made such that as $\tilde{\mathscr{P}}_T\to\mathscr{P}_T$ in $C^{\infty}$,  $\tilde{H}\to H$ in $C^{\infty}$ and $\tilde{H}-H$ is supported on an arbitrarily small neighborhood of $\Phi^{T/2}_H(x)$ .
\end{proposition}

\begin{proof}
Denote by $\alpha:=\Phi^{T/2}_H(x)$. By Theorem \ref{flowbox}, all Hamiltonian flows are locally conjugate to the Hamiltonian flow $\Phi^t_{p_n}$, thus we can find a small neighborhood $U\subseteq \Omega$ of $\alpha$  and a symplectic coordinate $(\hat{\textbf{q}},\hat{\textbf{p}})$ in $U$ such that the Hamiltonian flow $\Phi^t_H$ in $U$ is conjugate to $\Phi^t_{H^0}$ with $H^0=(\hat{\textbf{p}}^2-1)/2$. In particular, the restriction of $\Phi^t_H$ on $U\cap H^{-1}(0)$ is conjugate to the restriction of  $\Phi^t_{H^0}$ on $\hat{\textbf{p}}^2=1$, which is the geodesic flow with Euclidean metric on the unit cotangent bundle. 

We denote by $\Pi$ the projection in $U$ to its $\hat{\textbf{q}}$ coordinates and $c_x(t)$ the orbit segment between $x$ and $\Phi^t_H(x)$. Since $U$ is open, we can find  $T_0>0$ such that 
$$c_{\alpha}(T_0)\subseteq U\cap (H^0)^{-1}(0).$$
Let $D$ be the Euclidean $n$-disc in $\Pi(U)$ such that $\Pi(c_{\alpha}(T_0))$ is a diameter of $D$. Let $\sigma: U^*_{in}\to U^*_{out}$ be the dual lens map of $D$. By taking $V$ sufficiently small we may assume all orbits starting from $V$ pass through $U^*_{in}$ before time $T$. Hence any $C^{\infty}$ small symplectic perturbation of the Poincar\'{e} map $\mathscr{P}_T$ on $V$ is equivalent to a $C^{\infty}$ symplectic local perturbation of $\sigma$, which can be realized as a local perturbation of the Euclidean metric in $D$ via Proposition \ref{maincor} and Remark \ref{localpert}. Such pertubation of $H^0$ gives the desired perturbation of $H$.
\end{proof}

\section{Proof of Theorems \ref{clslem}, \ref{locclslem} and \ref{hamclosing}}
In this section we will use the results in Section 2 and 3 to prove the main theorems.
\begin{proof}[Proof of Theorem \ref{locclslem}]
Denote by $x_0:=\pi(g^{\varphi}_{-2\rho}(v))$ and $D:=\overline{B^+_{\varphi}(x_0, \rho)}$. Recall that $\rho$ is the radius in Lemma \ref{radlem} thus $(D,\varphi)$ is simple. Let $\sigma: U^*_{in}\rightarrow U^*_{out}$ be the dual lense map of $(D,\varphi)$. Denote by $U^*_0:=g^{\varphi}_{-2\rho}(\mathscr{L}(U))$ and $U^*_+$ (resp. $U^*_-$) the first intersection of $U^*_{out}$ (resp. $U^*_{in}$) with forward (resp. backward) $g^{\varphi}_t$-orbits of $U^*_0$  . 

By choosing $U$ sufficiently small we may assume 

(1) $\pi(U) \cap D=\emptyset$.

(2) $\overline{\pi(U^*_0)}\subseteq B(x_0, \rho/2)$.

(3) $\pi(U^*_+)$ and $\pi(U^*_-)$ are disjoint and $\pi(U^*_+\cup U^*_-)\neq \partial D$.

Let us take $h\in C^{\infty}(M,\mathbb{R}_{\geq 0})\backslash\{0\}$ with small $C^{\infty}$ norm, and supp $h\subseteq U^*_0$.
By Lemma \ref{echlem}, there exists $t\in [0,1]$ and $\gamma\in\mathcal{P}(U^*_\varphi \Sigma,(1+th)\lambda)$ which intersects supp $h$. As supp $h$ is a compact subset of $U_0$,  there exists $\alpha_0\in(U^*_0\backslash \text{supp }h)\cap\gamma$. Since the Reeb flows on $(U^*_\varphi \Sigma,\lambda)$ and $(U^*_\varphi \Sigma,(1+th)\lambda)$ agree outside supp $h$, $\alpha':=g^{\varphi}_{2\rho}(\alpha_0)\in\gamma$. On the other hand, $\alpha_0\in U^*_0$ implies $ \alpha'\in \mathscr{L}(U)$. Thus $\alpha'\in \gamma\cap \mathscr{L}(U)$.
However, the Reeb flow on $(U^*_\varphi \Sigma,(1+th)\lambda)$ may not be a geodesic flow. In order to get a geodesic flow we need to apply Corollary \ref{maincor}.

$(1+th)\lambda$ agrees with $\lambda$ outside $U^*_0$, thus the symplectic form $\omega$ on $U^*_{in}\cup U^*_{out}$ remains unchanged under perturbation. Hence the Reeb flow on $(U^*_\varphi \Sigma,(1+th)\lambda)$ induces a symplectic map $\tilde{\sigma}: U^*_{in}\to U^*_{out}$. It is clear that supp$(\tilde{\sigma}-\sigma)\subseteq U^*_-$. Since the condition ($\ast$) is satisfied by (3), by applying Corollary \ref{maincor}, we can get a simple Finsler disc $(D,\tilde{\varphi})$ with dual lens map $\tilde{\sigma}$. As $h\to 0$ in $C^{\infty}$, we have $\tilde{\sigma}\to\sigma$ in $C^{\infty}$ and hence $\tilde{\varphi}\to\varphi$ in $C^{\infty}$. The Finsler surface $(\Sigma, \tilde{\varphi})$ is obtained by gluing $(D, \tilde{\varphi})$ with $(\Sigma\backslash D, \varphi)$. The $\tilde{\varphi}$-geodesic $\tilde{\gamma}$ starting with $\alpha'$ is periodic since it coincides with $\gamma$ outside $D$. The $\tilde{\varphi}$-geodesic in the tangent bundle starting with $v'=\mathscr{L}^{-1}(\alpha')$ is the closed geodesic we need in Theorem \ref{locclslem}.

For reversible Finsler metrics, notice that $\tilde{\sigma}-\sigma$ is supported on $U^*_-$. We define a map $\hat{\sigma}: U^*_{in}\to U^*_{out}$ via
$$\hat{\sigma}(\alpha)=\left\{
\begin{aligned}
&-\tilde{\sigma}^{-1}(-\alpha),& &\text{ if } \alpha\in -U_0^+; \\
&\tilde\sigma(\alpha), & &\text{ otherwise.}
\end{aligned}
\right.$$
If $\sigma$ is reversible, so does $\hat{\sigma}$. By applying Corollary \ref{maincor}, we can get a reversible simple Finsler disc $(D,\hat{\varphi})$ with dual lens map $\hat{\sigma}$. 
\end{proof}
In order to prove Theorem \ref{clslem} we need the following perturbation lemma from Newhouse \cite{N77}(see also \cite{X96}):
\begin{lemma}\label{pertlem}
Let $(M,\omega)$ be an $n$-dimensional compact symplectic manifold with metric $d$ induced by some Riemannian structure. There exist $\epsilon_0$ and $c>0$, such that for any $0<\delta, \epsilon\leq\epsilon_0, r\geq 1$ and $x,y\in M$ with $d(y,x)<c\delta^r\epsilon$, there is a $\psi\in \text{Diff}^{\infty}_{\omega}(M), ||\psi-id||_{C^r}<\epsilon$ such that $\psi(x)=y$ and $\psi(z)=z$ for all $z\notin B_{\delta}(x)$.
\end{lemma}

\begin{remark}
Notice that in the orginal version in \cite{N77}, they got $\psi\in \text{Diff}^{r}_{\omega}(M)$. However, the proof still holds if we replace $C^r$ by $C^{\infty}$ and the unperturbed diffeomorphism (identity map in our case) is $C^{\infty}$.
\end{remark}
\begin{proof}[Proof of Theorem \ref{clslem}]

Let $d$ be any metric induced by a Riemannian structure on $U_{\varphi}\Sigma$.  Denote by $D^{\pm}:=D^{\pm}_{\varphi}(v,\rho)$, $\alpha:=\mathscr{L}(v)$ and $\alpha^-:=g^{\varphi}_{-2\rho}(\alpha)$. Since $D^{\pm}$ are simple we use $\sigma^{\pm}$ to denote the dual lens maps of $(D^{\pm}, \varphi)$. Let $d^{-}_{in}$ be the intrinsic metric induced by $d$ on $(D^-)^*_{in}$ and denote by $B^{-}_{in}(\alpha^-,\delta):=\{\alpha'\in (D^-)^*_{in}: d^{-}_{in}(\alpha^-, \alpha')<\delta\}$. When $\delta$ is small we can define a map $\eta: \sigma^-(B^{-}_{in}(\alpha^-,\delta))\to (D^+)^*_{in}$ by mapping $\beta\in \sigma^-(B^{-}_{in}(\alpha^-,\delta))$ to the first intersection of $(D^+)^*_{in}$ with the forward $g^{\varphi}_t$-orbit of $\beta$. It is clear that $\eta(\alpha)=\alpha$ and $\eta$ is $C^{\infty}$ when $\varphi$ is $C^{\infty}$.

Let $c$ and $\epsilon_0$ be the constants for $M=(D^-)_{in}^*$ in Lemma \ref{pertlem}. From now on we fix $\delta<\epsilon_0$.  For any $\epsilon<\epsilon_0$, by Theorem \ref{locclslem}, there exists $C^r$-perturbation $\varphi_1$ of $\varphi$ and $\beta^-\in B^{-}_{in}(\alpha^-,c\delta^r\epsilon)$ so that the $g^{\varphi_1}_t$-orbit of $\beta^-$ is closed. Notice that by construction $\varphi_1$ agrees with $\varphi$ on $D^{\pm}$ hence $(D^{\pm})^*_{in}$ and $(D^{\pm})^*_{out}$ remain untouched under perturbation.

By Lemma \ref{pertlem}, one can find $\psi\in \text{Diff}^{\infty}_{\omega}((D^-)_{in}^*)$ with $||\psi-id||_{C^r}<\epsilon$, $\psi(\alpha^-)=\beta^-$ and $\psi=id$ outside $B^{-}_{in}(\alpha^-,\delta)$. Define $\tilde{\sigma}^-: (D^-)_{in}^*\to (D^-)_{out}^*$ by $\tilde{\sigma}^-:=\sigma^-\circ \psi^{-1}$. $\tilde{\sigma}^-$ is a $C^r$-perturbation of $\sigma^-$ supported on $B^{-}_{in}(\alpha^-,\delta)$. By Corollary \ref{maincor}, there exists a $C^{\infty}$ Finsler metric ${\varphi}_-$ on $D^-$, $C^r$ closed to $\varphi$, such that $\sigma^-$ is the dual lens map of ${\varphi}_-$. Moreover the $g^{{\varphi}_-}_t$-orbit of $\beta^-$ passes through $\alpha$ since $\tilde{\sigma}^-(\beta)=\alpha$.

Now we perturb the metric in $D^+$. Define $\tilde{\sigma}^+: \eta\tilde{\sigma}^-(B^{-}_{in}(\alpha^-,\delta))\to (D^+)_{out}^*$ by 
$$\tilde{\sigma}^+:=\sigma^+\eta\sigma^-(\tilde{\sigma}^-)^{-1}\eta^{-1}.$$ 
Notice that $\tilde{\sigma}^+$ is a $C^r$ perturbation of $\sigma^+$ and $\tilde{\sigma}^+$ agrees with $\sigma^+$ on the boundary. Thus we can glue $\tilde{\sigma}^+$ with $\sigma^+$ to get $\tilde{\sigma}^+: (D^+)_{in}^*\to (D^+)_{out}^*$. Apply Corollary \ref{maincor} again we get a $C^{\infty}$ Finsler metric ${\varphi}_+$ on $D^+$, $C^r$ closed to $\varphi$ and the associated dual lens map is $\tilde{\sigma}^+$. Now we use $\tilde{\varphi}$ to denote the metric attained by gluing ${\varphi}_1$ with ${\varphi}_+$ and ${\varphi}_-$. Notice that $\tilde{\sigma}^+\eta\tilde{\sigma}^-=\sigma^+\eta\sigma^-$ on $B^{-}_{in}(\alpha^-,\delta)$, thus the   $g^{\tilde{\varphi}}_t$-orbit of $\beta^-$ is still closed and it passes $\alpha$.

The reversible case is similar to the counterpart in the proof of Theorem \ref{locclslem}.
\end{proof}

Similarly, in order to prove Theorem \ref{hamclosing},  we have only to prove the corresponding  local closing lemma:
\begin{proposition}\label{hamlocclosing}
Let $(\Omega, \omega), H, h$ be the same as in the assumption of Theorem \ref{hamclosing},  then for any open set $U\subseteq H^{-1}(h)$, one can make a $C^{\infty}$-small perturbation $\tilde{H}$ of $H$ such that there is a closed $\Phi^t_{\tilde{H}}$-orbit passing through $U$.
\end{proposition}
\begin{proof}[Proof of Proposition \ref{hamlocclosing}]
The idea is similar to the proof of Theorem \ref{locclslem}. Let $\lambda$ be the contact 1-form on $Y:=H^{-1}(h)$ and $\varphi_t$ be the associated Reeb flow. 

By shrinking $U$ if necessary, we can find $\tau>0$ so that $\Phi^{-\tau}_H(U)\cap U=\emptyset$. Take any $x_0 \in \Phi^{-\tau}_H(U)$, we can find two points $x_i:=\Phi^{t_i}_H(x_0), i=1,2$, so that $t_1<0<t_2<\tau$ and $x_1,x_2\notin U\cup \Phi^{-\tau}_H(U)$. Let $\Sigma_1, \Sigma_2$ be two transversal sections containing $x_1, x_2$ respectively with $(\Sigma_1\cup\Sigma_2) \cap(U\cup \Phi^{-\tau}_H(U))=\emptyset$.  Denote by $\mathscr{P}: \Sigma_1\to \Sigma_2$ the  Poincar\'{e} map with respect to $\Phi^t_H$. Since $d\lambda=\omega$ and $\varphi_t$ is orbit equivalent to $\Phi^t_H$ on $Y$, $\mathscr{P}$ is also the  Poincar\'{e} map with respect to $\varphi_t$. See the following figure.

\begin{figure}[htbp] 
\centering\includegraphics[width=5in]{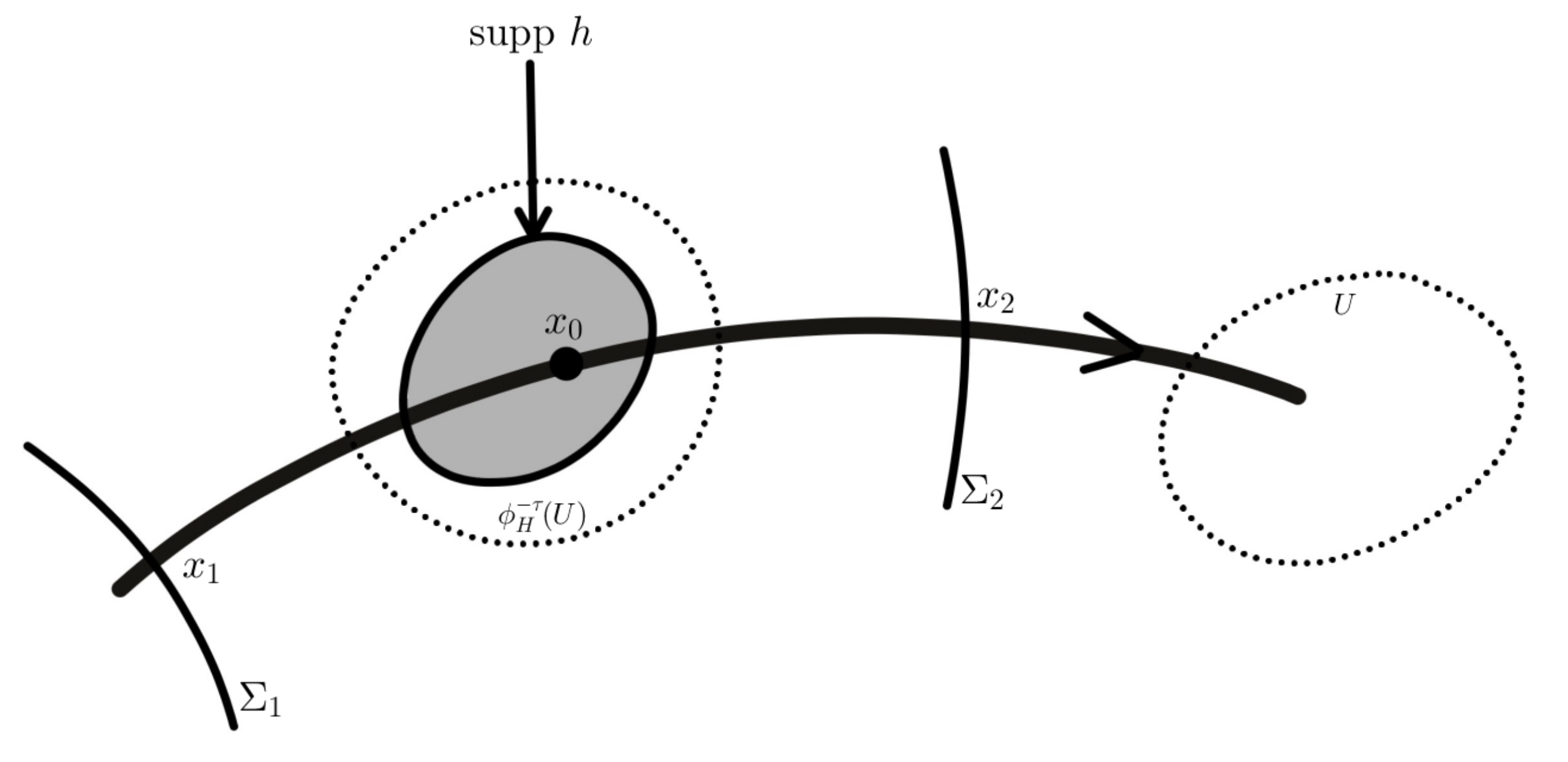} 
\caption{}\label{fig:2} 
\end{figure} 

Take $h\in C^{\infty}(Y,\mathbb{R}_{\geq 0})$ with $x_0\in\text{supp }h\subseteq \Phi^{-\tau}_H(U)$. By Lemma \ref{echlem}, we can perturb $\lambda$ in $C^{\infty}$ so that the resulting Reeb flow on $(Y,\tilde{\lambda})$ has a periodic orbit through $U$. The resulting Poincar\'{e} map $\tilde{\mathscr{P}}: \Sigma_1\to \Sigma_2$ is a $C^{\infty}$ small perturbation of $\mathscr{P}$ on a small subset of $\Sigma$ given both $C^{\infty}$ norm of $h$ and $\text{supp }h$ are sufficiently small. With Proposition \ref{pertham} we get a $C^{\infty}$ perturbation $\tilde{H}$ of $H$ so that the associated Poincar\'{e} map is $\tilde{\mathscr{P}}$.  It is clear that the $\Phi^t_{\tilde{H}}$ orbit through $\Phi^{\tau}_H(x_0)$ is periodic.

\end{proof}

Theorem \ref{hamclosing} can be derived from Proposition \ref{hamlocclosing} in a similar way as in the proof of Theorem \ref{clslem}.

\section{Examples}
\subsection{Geodesics on Finsler $S^2$}  It is well-known that a Riemannian 2-sphere has infinitely many distinct prime closed geodesics \cite{F92}\cite{B93}. On the other hand, if one considers a Finsler 2-sphere, the famous example constructed by A. Katok\cite{K73} shows that an irreversible Finsler $S^2$ may have exactly two closed geodesics. The results in Hofer-Wysocki-Zehnder \cite{HWZ03} and \cite{CGHP17} imply that a generic Finsler $S^2$ has either two or infinitely many closed geodesics. By taking $\Sigma=S^2$ in Theorem \ref{dencor}, we can improve the generic result with the following corollary:
\begin{corollary}
For any $r\geq 1$, a $C^r$-generic Finsler $S^2$ has infinitely many distinct prime closed geodesics. 
\end{corollary}
In particular, one can make a $C^{\infty}$ perturbation of Katok's Finsler $S^2$ to get a dense set of periodic geodesics in the unit tangent bundle.

\subsection{Dynamics between KAM tori} Another interesting application of Theorem \ref{dencor} is when $\Sigma$ is $\mathbb{T}^2=\mathbb{R}^2/\mathbb{Z}^2$. By KAM theorem any $C^r$-small $(r\geq 5)$ perturbation $\varphi$ of a flat metric $\varphi_0$ still has a large family of invariant tori in the unit cotangent bundle. The dynamics on these tori are quasi-periodic but not periodic. Since the KAM tori are of codimension 1 in the energy surface, they separate the unit cotangent bundle. It was proved by Conley-Zehnder  \cite{CZ83} that KAM tori are approximated by periodic orbits. Theorem \ref{dencor} and Theorem \ref{hamclosing} show that for  $C^{\infty}$-generic Finsler metric (or a $C^{\infty}$-generic Hamiltonian flow on $T^*\mathbb{T}^2$) near any flat $\mathbb{T}^2$, the closure of all periodic orbits fills in the entire energy surface.

\subsection{Level sets of contact type} 

A hypersurface $S\subseteq \Omega$ is of contact type if and only if it is traverse to a vector field on $\Omega$ which preserves $\omega$ (cf. Lemma 2 in \cite{W79}). This criterion allows us to verify the following examples without finding a contact 1-form:

\begin{itemize}
\item If $\Omega=\mathbb{R}^4$, any star-shaped (namely, transverse to the radical vector field) compact hypersurface is of contact type \cite{W79}. In particular, any convex hypersurface is of contact type.

\item If $\Omega=T^*\Sigma$ for a closed surface $\Sigma$, let $S$ be a hypersurface whose projection onto $\Sigma$ is proper. Then $S$ is of contact type if and only if its intersection with each cotangent space $T^*_x\Sigma$ is star-shaped \cite{W79}. For example, the level sets of a Tonelli Hamiltonian are of contact type.

\item If $\Omega=T^*\Sigma$ for a Riemannian surface $\Sigma$ and $H$ is a classical Hamiltonian, namely, $H=K+V$ where $K$ is kinetic energy and $V$ is a potential on $\Sigma$, then any compact regular energy surface of $H$ is of contact type (see \cite{HZ11} Chapter 4.4). 
\end{itemize}

\bibliography{closinglemma}
\bibliographystyle{amsalpha}

\end{document}